\documentclass{article}
\usepackage{graphicx} 
\usepackage{amsmath, amssymb,amsthm}
\usepackage{pgf,tikz,calc}
\usepackage{enumitem}
\usepackage[colorlinks=true, linkcolor=blue]{hyperref}
\usepackage[normalem]{ulem}
\usepackage{comment}
\newtheorem{theorem}{Theorem}
\newtheorem{proposition}[theorem]{Proposition}

\newtheorem{corollary}[theorem]{Corollary}
\newtheorem{remark}[theorem]{Remark}
\newtheorem{observation}[theorem]{Observation}
\newtheorem{claim}{Claim}

\newtheorem{open}{Open problem}
\newtheorem{conjecture}[open]{Conjecture}
\usetikzlibrary{calc}
\usepackage[top=1.5in, bottom=1.75in, left=1.55in, right=1.55in]{geometry}

\usepackage{amsthm}
\theoremstyle{definition}
\newcommand{\smallqed}{{\tiny ($\Box$)}}
\newcommand{\mc}{\overline{\chi}_{\geqslant}}
\newcommand{\cg}{C_{\geqslant}}
\def\cp{\,\square\,}

\newcommand{\cP}{{\cal P}}

\newcommand{\cI}{{\cal I}}

\setlist[description]{style=nextline}
\newcommand{\cs}[1]{{\bf \textcolor{blue}{#1}}}

\usepackage[dvipsnames]{xcolor}
\DeclareMathOperator{\dist}{dist}

\date{}

\begin{document}

\title{Majority C-coloring in Cartesian products}
\author{
Csilla Bujt\'as $^{a,b,}$\thanks{Email: \texttt{csilla.bujtas@fmf.uni-lj.si}}
\and
Magda Dettlaff $^{c,}$\thanks{Email: \texttt{magda.dettlaff@ug.edu.pl}}
\and
Hanna Furma\'nczyk $^{c,}$\thanks{Email: \texttt{hanna.furmanczyk@ug.edu.pl}}
\and
Aleksandra Laskowska $^{c,}$\thanks{Email: \texttt{aleksandra.laskowska@ug.edu.pl}}
\and
Zsolt Tuza $^{d,e,}$\thanks{Email: \texttt{tuza.zsolt@mik.uni-pannon.hu}}
}
\date{}
\maketitle

\begin{center}
$^a$ Faculty of Mathematics and Physics, University of Ljubljana, Slovenia\\
\medskip

$^b$ Institute of Mathematics, Physics and Mechanics, Ljubljana, Slovenia\\
\medskip

$^c$ Faculty of Mathematics, Physics and Informatics, University of Gda\'nsk, Poland\\
\medskip

$^d$ HUN-REN Alfr\'ed R\'enyi Institute of Mathematics, Budapest, Hungary\\
\medskip

$^e$ Department of Computer Science and Systems Technology,\\ University of Pannonia, Veszpr\' em, Hungary\\
\medskip

\end{center}
\maketitle

\begin{abstract}
A majority C-coloring of a graph $G$ assigns colors to the vertices such that every vertex shares its color with at least half of its neighbors. The maximum number of colors that can be used in such a coloring of $G$ is denoted by  $\overline{\chi}_{\geqslant}(G)$. In this paper, the focus is on the majority C-coloring in Cartesian product graphs. It is shown that $\overline{\chi}_{\geqslant}(G \square H) \ge \overline{\chi}_{\geqslant}(G) \overline{\chi}_{\geqslant}(H)$ gives a sharp lower bound, but the difference also can be arbitrarily large. For two-dimensional Hamming graphs, the exact value  $\overline{\chi}_{\geqslant}(K_m \square K_n) = \min\{m,n\}$ is established. Balanced Hamming graphs of higher dimension, that is the $k$th powers of complete graphs with respect to the Cartesian product, are also studied. It is proved that $\overline{\chi}_{\geqslant}(K_n^{\square, k})= n^{k/2}$ holds for every even integer $k$. If $k$ is a fixed odd integer and $n \rightarrow \infty$, then the exact order of growth is determined as $\overline{\chi}_{\geqslant}(K_n^{\square, k})= \Theta(n^{\lfloor k/2 \rfloor})$.
Further, if the Hamming graph is the $k$-dimensional hypercube, then $\overline{\chi}_{\geqslant}(K_2^{\square, k})= 2^{\lfloor k/2\rfloor}$. On the other hand, a majority C-coloring of $K_n^{\square, k}$ with at least $3 n^{\lfloor k/2\rfloor}/2 $ colors is presented for every $n \ge 7$ and odd $k \ge 3$. For Cartesian grids,  the main result shows that $\overline{\chi}_{\geqslant}(P_m \square P_n) = 1 + \lfloor m/2\rfloor \lfloor n/2\rfloor$ if at least one of $m$ and $n$ is odd, while $\overline{\chi}_{\geqslant}(P_m \square P_n)=mn/4$ holds if both parameters are even and $m \ge n \ge 4$. The paper concludes with  several open problems.
\end{abstract}
\medskip

\medskip

\noindent
{\bf AMS Subj.\ Class.\ (2020):}  05C15, 05C76

\maketitle

\section{Introduction}
\label{sec:pre}
A majority C-coloring, as introduced in~\cite{majority-1}, is a vertex coloring of a graph $G$ so that every vertex shares its color with at least half of its neighbors. The majority C-chromatic number $\mc(G)$ is the maximum number of colors that can be used in such a coloring of $G$. In this paper, we focus on an essential graph operation and study the majority C-chromatic number of Cartesian products, with particular attention to Hamming graphs and Cartesian grids.

Hamming graphs are among the most extensively studied Cartesian product graphs. Their rich combinatorial structure and high symmetry have made them a central object of research in graph theory, coding theory, and interconnection networks.

Since grids constitute fundamental class of structured graphs, they provide a natural starting point for investigating the behavior of the proposed coloring model.

\subsection{Standard definitions}
We consider only simple undirected graphs. For a graph $G$, if not indicated differently, we refer to its vertex set and edge set as $V(G)$ and $E(G)$, respectively. The \emph{open neighborhood} of a vertex $v \in V(G)$, denoted by $N_G(v)$, is the set of the neighbors of $v$, while its \emph{closed neighborhood} is $N_G[v]=N_G(v) \cup \{v\}$. As usual, $\deg_G(v)$ is the degree of a vertex $v \in V(G)$, while $\delta(G)$, $\Delta(G)$ stand for the minimum and maximum vertex degree in $G$, respectively. Two vertices $u$ and $v$ are called \emph{true twins} in a graph $G$ if $N_G[u]=N_G[v]$. We note that $P_n$ and $C_n$ denote, respectively, the path graph and the cycle graph on $n$ vertices. For a positive integer $k$, the symbol $[k]$ stands for the set $\{1,2, \dots, k\}$.  
\medskip

For two graphs $G$ and $H$, the \emph{Cartesian product} $G \cp H$ is defined on the vertex set $V(G) \times V(H)$ so that two vertices $(g,h)$ and $(g', h')$ are adjacent if $g=g'$ and $hh' \in E(H)$, or $h=h'$ and $gg' \in E(G)$. For a vertex $g \in V(G)$, the set of vertices $^gH=\{(g,y)\colon y\in V(H)\}$ is an \emph{$H$-layer} in $G \cp H$. We use the same notation $^gH$ to refer to the subgraph induced by this layer. By definition, $^gH \cong H$ for every $g \in V(G) $. The \emph{$G$-layer} $G^h$ is defined analogously for every fixed $h \in V(H)$.

In general, a \emph{Hamming graph} $F$ is the Cartesian product of some complete graphs. If $F$ is the product of $k$ factors, that is $F=K_{m_1} \cp K_{m_2} \cp \cdots \cp K_{m_k}$, then it is a Hamming graph of \emph{dimension $k$}. If  all the $k$ factors are isomorphic to $K_n$, the resulting graph is the \emph{balanced Hamming graph} $K_n^{\cp, k}$ which is also called the \emph{$k^{\rm th}$ power} of the graph $K_n$ (with respect to the Cartesian product). 
Then the $k$-dimensional hypercube $Q_k$ is the balanced Hamming graph $K_2^{\cp, k}$. We also note that $K_n^{\cp, k}$ is a $k(n-1)$-regular graph of order $n^k$ and diameter $k$.

A \emph{Cartesian grid} is the Cartesian product of two paths. We say that a vertex $v \in V(G)$ is a \emph{corner vertex}, a \emph{border vertex}, or an \emph{inner vertex} of the grid $G=P_m \cp P_n$ if $\deg_G(v)$ equals $2$, $3$, or $4$, respectively. Two corner vertices are \emph{opposite} if their distance is $m+n-2$.

\medskip

For additional graph theoretic definitions and facts related to graph products, we refer the reader to the books~\cite{West-book} and~\cite{Product-book}.

\subsection{Majority C-coloring}
For a graph $G$, a surjective mapping $\varphi \colon V(G) \rightarrow [p]$  is a \emph{majority C-coloring} (or shortly, \emph{$\cg$-coloring}) of $G$ with color classes $V_1=\varphi^{-1}(1), \dots, V_p=\varphi^{-1}(p)$, if every vertex $v \in V(G)$ satisfies the following \emph{majority condition}: 
\begin{equation} \label{eq:majority-1}
  v \in V_i \quad \Rightarrow \quad | N_G(v) \cap V_i| \ge \frac{1}{2} \deg_G(v).
\end{equation} 
The maximum number of colors that can be used in a $\cg$-coloring of $G$ is called the \emph{majority C-chromatic number $\mc(G)$} (or shortly, \emph{$\cg$-chromatic number})
 of the graph. 

When a $\cg$-coloring $\varphi$ is given, we will refer to its color classes as $V_1, \dots , V_p$, and assume that neither of them is empty. Further, we say that $\varphi$ is a \emph{$\mc$-coloring} of $G$ if it uses exactly $\mc(G)$ colors. For illustration see Fig.~\ref{fig:2-dim-Ham}.
 
\medskip

The following properties follow directly from the definition of $\cg$-coloring and $\cg$-chromatic number. They were first observed in~\cite{majority-1}.
\begin{observation} \label{obs:basic} Let $\varphi$ be a majority C-coloring of the graph $G$.
\begin{enumerate}[label=\normalfont(\roman*)]
        \item If $\varphi$ is a $\mc$-coloring of a graph $G$, then every color class induces a connected subgraph in $G$.
    \item For the complete graph $K_n$, $n\geq 1$, it holds that $\mc(K_n)=1$.    
    \item For the cycle $C_n$ and the path $P_n$, $n\geq 3$, it holds that $\mc(C_n)=\mc(P_n)=\lfloor \frac{n}{2}\rfloor$.
\end{enumerate}   
\end{observation}

\subsection{Motivation}
The definition of the majority C-coloring combines the coloring conditions of the \emph{majority colorings} of graphs \cite{Anholcer, Bock, Kalinowski, KeOuSeZyWo-17} with the idea behind the \emph{C-colorings} of hypergraphs~\cite{Bujtas-2009, BuTu-2010, Jiang-2002, Vol-1995, Vol-2002}. For definitions and a short summary, we refer the reader to the introductory paper~\cite{majority-1}. 

Another coloring model related to the $\cg$-coloring is the \emph{$k$-improper C-coloring} of graphs~\cite{BuSaTuPuVa-10}. If $k$ is a fixed positive integer, then it is a vertex coloring of a graph $G$ so that every vertex $v \in V(G)$ shares its color with all but at most $k$ of its neighbors. The maximum number of colors in such a coloring of $G$ is denoted by $\overline{\chi}_{k \mbox{\footnotesize{-imp}}}(G)$. By this definition, the majority C-colorings and the $\lfloor d/2 \rfloor$-improper C-colorings coincide over the class of $d$-regular graphs. Therefore, if $G$ is $d$-regular, then 
\begin{equation} \label{eq:k-improper}
    \mc(G)=\overline{\chi}_{\lfloor \frac{d}{2}\rfloor \mbox{\footnotesize{-imp}}}(G).
\end{equation}
 
A $\cg$-coloring of a graph $G$ with two colors defines a partition of $V(G)$ into two color classes $V_1$ and $V_2$ such that $|N(v_1) \cap V_1| \ge |N(v_1) \cap V_2|$ holds for every $v_1 \in V_1$, and  $|N(v_2) \cap V_2| \ge |N(v_2) \cap V_1|$ holds if $v_2 \in V_2$. In the earlier literature, these partitions were referred to as \emph{satisfactory partitions}~\cite{BTV-06, GeKob-2004, Shaf} or \emph{internal partitions}~\cite{Anastos, Ban}. However, those studies focused only on partitions into two (or three) classes and mainly considered them from an algorithmic perspective. In contrast, we are interested in the coloring aspect of this problem and concentrate on the maximum number of classes in such partitions. 

As for applications of the majority C-coloring model, the introductory paper~\cite{majority-1} discussed connections to public service allocation problems and Schelling's model of segregation. Further possible applications were pointed out in the works~\cite{Anastos, Ban, BTV-06} in terms of internal (or satisfactory) partitions. 
\medskip

The introductory paper~\cite{majority-1} explored the behavior of the majority C-chromatic number over some basic graph classes like trees, cubic graphs, powers of paths and cycles. It also investigated various interrelationships among the order, size, classical chromatic number, and $\cg$-chromatic number of a graph. It was pointed out that $\mc(G)$ is not monotone under taking (induced) subgraphs. In fact, for an edge $e$ of $G$, both the lower and upper bounds in
$$\mc(G)-2 \leq \mc(G-e) \leq \mc(G)+1
$$
are sharp. The present work naturally extends this line of research by exploring the behavior of the majority C-chromatic number in Cartesian products of graphs.

\subsection{Results and structure of the paper}
In Section~\ref{sec:Cart}, we first prove two general lower bounds on $\mc(G \cp H)$. The subsequent results of the paper provide infinitely many examples of their sharpness. Then, we turn our attention to Hamming graphs by establishing that
$$\mc(K_m \cp K_n) = \min\{m,n\}$$ holds for every two-dimensional Hamming graph. For Hamming graphs of higher dimension, the problem seems much more difficult, but we can prove the exact value 
 $$\mc (K_n^{\cp, k})= n^{\frac{k}{2}}$$
 for the balanced Hamming graphs of dimension $k$, if the exponent is even. For an odd $k$, the  inequality
 \begin{equation} \label{eq:odd-exponent}
     \mc (K_n^{\cp, k})\ge  n^{\lfloor\frac{k}{2} \rfloor}
 \end{equation} 
  is always valid, and it holds with equality in the case of a hypercube that is, when $n=2$. On the other hand, if $k\ge 3$ is odd and $n \ge 7$ or $n=5$, a $\cg$-coloring of $K_n^{\cp, k}$ that uses at least $3 n^{\lfloor k/2\rfloor}/2 $ colors is presented, thereby improving the lower bound in~\eqref{eq:odd-exponent}. On the other hand, although the exact formula $n^{\lfloor \frac{k}{2}\rfloor}$ fails in odd dimensions, it remains correct at the asymptotic level. 
  Theorem~\ref{thm:asymptotic} demonstrates that if $k$ is a fixed odd integer and $n \rightarrow \infty$, then
  $$    \mc (K_n^{\cp, k})= \Theta(n^{\lfloor\frac{k}{2} \rfloor}).
  $$
  \medskip

  Section~\ref{sec:grids} is devoted to the study of $\cg$-colorings of Cartesian grids. If both $m$ and $n$ are even and $\min\{m,n\}\ge 4$, the formula
  \begin{equation*}
      \mc(P_m \cp P_n)= \Bigl\lfloor\frac{m}{2}\Bigr\rfloor\Bigl\lfloor\frac{n}{2}\Bigr\rfloor
  \end{equation*}
  is not hard to prove. However, when at least one of $m$ and $n$ is odd, and \cs{$\min\{m,n\}\ge 3$}, the proof of the formula
  \begin{equation*}
      \mc(P_m \cp P_n)= 1+\Bigl\lfloor\frac{m}{2}\Bigr\rfloor\Bigl\lfloor\frac{n}{2}\Bigr\rfloor
  \end{equation*}
  needs a more intricate approach.
  \medskip

  Section~\ref{sec:concluding} provides a short overview of the open problems related to our work.

\section{Cartesian products and Hamming graphs}
\label{sec:Cart}
We begin studying $\cg$-colorings of Cartesian products by establishing general lower bounds on $\mc(G\cp H)$ and determining the exact value of $\mc(K_m \cp K_n)$ for two-dimensional Hamming graphs. We also derive exact values for the (Cartesian) powers of complete graphs if the exponent is even, and give estimates when it is odd.
\subsection{General lower bounds}
\label{subsec:cartesian-lower}
\begin{proposition} \label{prop:cart}
 If $G$ and $H$ are two graphs and $n_H$ denotes the order of $H$, the following statements hold.
  \begin{itemize} 
  \item[$(i)$] $\mc(G\cp H) \ge \mc(G)\, \mc(H)$.
  \item[$(ii)$] If $\delta(G) \ge \Delta(H)$, then $\mc(G\cp H) \ge n_H$.
    \end{itemize}
  \begin{proof}
  (i) We set $\mc(G)= t$ and $\mc(H)= s$. Let $V_1, \dots , V_t$ be the color classes of a $\mc$-coloring of $G$ and, similarly, let $U_1, \dots U_s$ be the color classes of a $\mc$-coloring of $H$. By definition, $|N_G(v) \cap V_i|\ge \frac{1}{2}\deg_G(v)$ holds for every $v \in V_i$, and $|N_H(u) \cap U_j|\ge \frac{1}{2}\deg_H(u)$ holds for every $u \in U_j$. In the Cartesian product $F=G\cp H$, consider the vertex set $V_i \times U_j$ for every $(i,j) \in [t]\times [s]$. These $ts$ sets together define a partition $\cP$ of $V(F)$. We prove that the corresponding coloring $\varphi$ is a $\cg$-coloring of $F$.

  Consider an arbitrary vertex $(v,u) \in V(F)$. Recall that $\deg_F(v,u)= \deg_G(v) + \deg_H(u)$. If $(v,u) \in V_i \times U_j$, then $v \in V_i$ and $|N_G(v) \cap V_i|\ge \frac{1}{2}\deg_G(v)$. Consequently, for the product $F$, the set $V_i \times U_j$ contains at least $\frac{1}{2}\deg_G(v)$ neighbors of $(v,u)$ of the form $(v,y)$. By symmetry, $V_i \times U_j$ contains at least $\frac{1}{2}\deg_H(u)$ neighbors of $(v,u)$ of the form $(x,u)$. Therefore, we have
  $$ |N_F((v,u)) \cap (V_i \times U_j)| \ge \frac{\deg_G(v)}{2} + \frac{\deg_H(u)}{2} = \frac{\deg_F((v,u))}{2}.
  $$
  Since the above argumentation is valid for every vertex from the product, $\cP$ defines a $\cg$-coloring with $ts= \mc(G)\, \mc(H)$ (non-empty) color classes and $\mc(G\cp H) \ge \mc(G)\, \mc(H)$ follows.  
  \medskip

  (ii) If $\delta(G) \ge \Delta(H)$, then every vertex $(v,u)$ of $F=G \cp H$ satisfies $$\deg_F((v,u)) = \deg_G(v)+ \deg_H(u) \leq 2 \deg_G(v).$$ Partition $V(F)$ into the $n_H$ classes $V(G) \times \{y\}$, one for each vertex $y\in V(H)$. Let $(v,u) \in V(F)$. Then $(v,u)$ belongs to the  class $V(G) \times \{u\}$ and 
  $$|N_F((v,u)) \cap (V(G) \times \{u\}) | \ge \deg_G(v) \ge \frac{\deg_F((v,u))}{2}.
  $$ 
  Since it holds for every vertex from $F$, the partition defined for $V(F)$ gives a $\cg$-coloring with $n_H$ colors. 
    \end{proof}
\end{proposition}
The remaining results of this (and the next) section demonstrate that the lower bounds in Proposition~\ref{prop:cart}(i) and (ii) are sharp for an infinite family of graphs. It will be discussed in more detail at the end of the section in  Remark~\ref{rmk:cart-tightness}.

\subsection{Two-dimensional Hamming graphs}
\label{subsec:2-dim-hamming}

A one-dimensional Hamming graph corresponds to a complete graph $K_m$ and, by Observation~\ref{obs:basic}(ii), $\mc(K_m)=1$ for every positive integer $m$. We now consider two-dimensional Hamming graphs (see Figure~\ref{fig:2-dim-Ham} for an illustration), and prove the following exact value for $\mc(K_m \cp K_n)$.

\begin{figure}[t!] 
        \centering
\begin{tikzpicture}
[scale=0.7,style=thick,x=1cm,y=1cm
]
\def\vr{5pt}
\begin{scope}[xshift=-1cm, yshift=0cm] 
\coordinate(u1) at (0.0,0.0);
\coordinate(u2) at (0.0,-2.0);
\coordinate(u3) at (0.0,-4.0);
\coordinate(v1) at (3.0,0.0);
\coordinate(v2) at (3.0,-2.0);
\coordinate(v3) at (3.0,-4.0);
\coordinate(w1) at (6.0,0.0);
\coordinate(w2) at (6.0,-2.0);
\coordinate(w3) at (6.0,-4.0);
\coordinate(z1) at (9.0,0.0);
\coordinate(z2) at (9.0,-2.0);
\coordinate(z3) at (9.0,-4.0);

\foreach \i in {1,2,3} 
{
\draw(v\i) -- (u\i) -- (w\i)--(z\i);
\draw(v1) -- (v2) -- (v3) ;
\draw(u1) -- (u2) -- (u3) ;
\draw(w1) -- (w2) -- (w3) ;
\draw(z1) -- (z2) -- (z3) ;
\draw plot [smooth, tension=1] coordinates {(u1) (0.5,-2) (u3)};
\draw plot [smooth, tension=1] coordinates {(v1) (3.5,-2) (v3)};
\draw plot [smooth, tension=1] coordinates {(w1) (6.5,-2) (w3)};
\draw plot [smooth, tension=1] coordinates {(z1) (9.5,-2) (z3)};
}
\draw plot [smooth, tension=1] coordinates {(u1) (3,0.4) (w1)};
\draw plot [smooth, tension=1] coordinates {(v1) (6,0.4) (z1)};
\draw plot [smooth, tension=1] coordinates {(v2) (6,-1.6) (z2)};
\draw plot [smooth, tension=1] coordinates {(v3) (6,-3.6) (z3)};
\draw plot [smooth, tension=1] coordinates {(u2) (3,-1.6) (w2)};
\draw plot [smooth, tension=1] coordinates {(u3) (3,-3.6) (w3)};
\draw plot [smooth, tension=1] coordinates {(u1) (4.5,0.8) (z1)};
\draw plot [smooth, tension=1] coordinates {(u2) (4.5,-1.2) (z2)};
\draw plot [smooth, tension=1] coordinates {(u3) (4.5,-3.2) (z3)};

\begin{scriptsize}
\foreach \i in {u,v,w,z} 
{
\draw(\i1)[fill=white] circle(\vr);
}
\foreach \i in {u,v,w,z} 
{
\draw(\i2)[fill=gray] circle(\vr);
}
\foreach \i in {u,v,w,z} 
{
\draw(\i3)[fill=black] circle(\vr);
}
\end{scriptsize}

\draw (-0.5, -0.4) node {$^{(1,3)}$};
\draw (-0.5, -2.4) node {$^{(1,2)}$};
\draw (-0.5, -4.4) node {$^{(1,1)}$};
\draw (2.5, -0.4) node {$^{(2,3)}$};
\draw (2.5, -2.4) node {$^{(2,2)}$};
\draw (2.5, -4.4) node {$^{(2,1)}$};
\draw (5.5, -0.4) node {$^{(3,3)}$};
\draw (5.5, -2.4) node {$^{(3,2)}$};
\draw (5.5, -4.4) node {$^{(3,1)}$};
\draw (8.5, -0.4) node {$^{(4,3)}$};
\draw (8.5, -2.4) node {$^{(4,2)}$};
\draw (8.5, -4.4) node {$^{(4,1)}$};

\end{scope}

\end{tikzpicture}
        \caption{A $\mc$-coloring of the two-dimensional Hamming graph $K_4 \cp K_3$ over the vertex set $[4]\times [3]$ that uses colors white, gray, and black.}
        \label{fig:2-dim-Ham}
\end{figure}
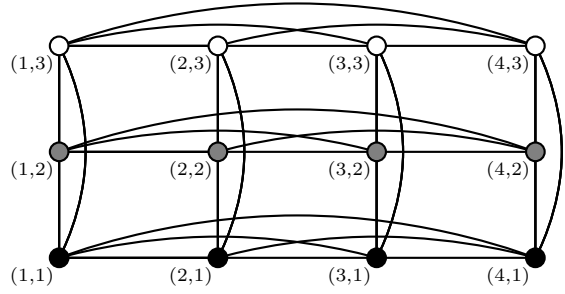
\begin{theorem} \label{thm:2-dim-Ham}
  For every two positive integers $m$ and $n$,  it holds that $$\mc(K_m \cp K_n) = \min\{m,n\}.$$ 
\end{theorem}
\begin{proof}
     By the commutativity of the Cartesian product, we may assume $m \ge n$. Then $\delta(K_m) = m-1 \ge \Delta(K_n)=n-1$, and Proposition~\ref{prop:cart} (ii) directly implies $\mc(K_m \cp K_n) \ge  n$.

     Let $G=K_m \cp K_n$ and $k = \mc(G)$. We consider a $\mc$-coloring of $G$ with color classes $V_1, \dots, V_k$ and want to prove that $|V_j| \ge m$ holds for each $j \in [k]$. The vertices of $G$ will be represented by two-dimensional vectors $(a,b) \in [m] \times [n]$. Two vertices are neighbors if they differ in exactly one coordinate, and $\deg_G(z)= m+n-2$ holds for every $z \in V(G)$. We also note that $G$ is a graph of diameter~$2$ whenever $n \ge 2$.

     Consider an arbitrary color class $V_j$ and a vertex $z_0=(a_0,b_0)$ with $z_0 \in V_j$. The neighborhood of $z_0$ can be partitioned into two classes 
     $$ Y= \{ (a_0,y) \colon y \in [n]\setminus \{b_0\}\} \quad \mbox{and} \quad 
     X= \{ (x,b_0) \colon x \in [m]\setminus \{a_0\}\}.
     $$
     Note that $|X|= m-1 $ and $|Y|=n-1$. By the majority condition~(\ref{eq:majority-1}),
     $$|V_j \cap (X \cup Y)| \ge \Bigl\lceil \frac{m+n-2}{2}\Bigr\rceil.$$
    \medskip
    
     We now distinguish two cases to show that $|V_j|\ge m$ holds for every color class.
     \begin{description}
         \item[Case 1. $V_j \cap Y \neq \emptyset$]
                Let $z_1$ be a neighbor of $z_0$ from $V_j \cap Y$. Since $z_1$ has only $n-2$ neighbors from $Y$ and no neighbor from $X$, we derive that $z_1$ has at least $\lceil \frac{1}{2}(m+n-2)\rceil -1-(n-2)$ neighbors in $V_j$ which are distance~$2$ apart from $z_0$. Then, $|V_j|$ can be estimated as follows.
          \begin{align*}
              |V_j| &\ge |\{z_0\}| + |V_j \cap (X\cup Y)| + \Bigl\lceil \frac{m+n-2}{2}\Bigr\rceil -1-(n-2)\\
              &\ge 1+ \Bigl\lceil \frac{m+n-2}{2}\Bigr\rceil + \Bigl\lceil \frac{m+n-2}{2}\Bigr\rceil -n +1\\
              &= 2 \Bigl\lceil \frac{m+n}{2}\Bigr\rceil -n \ge m.
          \end{align*}
             \item[Case 2. $V_j \cap Y = \emptyset$]
              Since $N(z_0)$ contains at least $\lceil \frac{1}{2}(m+n-2)\rceil$ vertices from $V_j$, we have $|X \cap V_j| \ge \lceil \frac{1}{2}(m+n-2)\rceil$ under the present assumption. Suppose, for a contradiction, that there exists a vertex $z' \in X \setminus V_j$. As $z_0$ is also from $V_j$ and $G[X]$ is a clique, vertex $z'$ has at least $\lceil \frac{1}{2}(m+n-2)\rceil +1$ neighbors in $V_j$.  Therefore,  more than half of the vertices in $N(z')$ are colored differently from $z'$, a contradiction. We infer that $X \cup \{z_0\} \subseteq V_j$, which itself proves $|V_j| \ge m$.
     \end{description}
      We could conclude $|V_j| \ge m$ in both cases. Then, since $G$ has $mn$ vertices, $\mc(G) \leq \frac{mn}{m}=n$ follows, which completes the proof of the theorem.     
\end{proof}

\section{Balanced Hamming graphs of higher dimension}
\label{subsec:hamming-power}
Recall that a balanced Hamming graph of dimension $k$ is the $k^{\rm th}$ (Cartesian) power $K_n^{\cp, k}=K_n \cp \dots \cp K_n$ for a positive integer $n$, where the product contains exactly $k$ factors. Balanced Hamming graphs are also well known for their use in cryptography and error-correcting codes. For an example, see Figure~\ref{fig:3-dim-Hamming}.

As $K_n^{\cp, 1} \cong K_1^{\cp, n} \cong K_n $, we already know that $\mc(K_n^{\cp, 1})= \mc(K_1^{\cp, n})=1$. Further, Theorem~\ref{thm:2-dim-Ham} implies that $\mc(K_n^{\cp, 2})=n$ holds for every positive integer $n$. In Section~\ref{subsec:even-dim}, we prove exact values for $\mc (K_n^{\cp, k})$ for every $n \ge 2$ and $k \ge 2$ if $k$ is even, while the main theorem in Section~\ref{subsec:odd-dim} determines the exact order of growth for $\mc(K_n^{\cp, k})$ when $k$ is a fixed odd integer and $n \rightarrow \infty$.
\begin{figure}[htb]
\centering
\begin{tikzpicture}[
  edgeI/.style={black, solid, line width=0.75pt},
  edgeJ/.style={black, dashed, line width=0.75pt},
  edgeK/.style={black, dotted, line width=0.75pt},
  vtx/.style={circle, draw=black, fill=white, line width=0.75pt, minimum size=5.6pt, inner sep=0pt}, scale=0.8
]
\draw[edgeI] (0.000,0.000) -- (2.295,0.000);
\draw[edgeI] (2.295,0.000) -- (4.590,0.000);
\draw[edgeI] (0.000,0.000) .. controls (1.530,0.227) and (3.060,0.227) .. (4.590,0.000);
\draw[edgeI] (0.000,2.295) -- (2.295,2.295);
\draw[edgeI] (2.295,2.295) -- (4.590,2.295);
\draw[edgeI] (0.000,2.295) .. controls (1.530,2.522) and (3.060,2.522) .. (4.590,2.295);
\draw[edgeI] (0.000,4.590) -- (2.295,4.590);
\draw[edgeI] (2.295,4.590) -- (4.590,4.590);
\draw[edgeI] (0.000,4.590) .. controls (1.530,4.817) and (3.060,4.817) .. (4.590,4.590);
\draw[edgeJ] (0.000,0.000) -- (0.000,2.295);
\draw[edgeJ] (0.000,2.295) -- (0.000,4.590);
\draw[edgeJ] (0.000,0.000) .. controls (-0.227,1.530) and (-0.227,3.060) .. (0.000,4.590);
\draw[edgeJ] (2.295,0.000) -- (2.295,2.295);
\draw[edgeJ] (2.295,2.295) -- (2.295,4.590);
\draw[edgeJ] (2.295,0.000) .. controls (2.068,1.530) and (2.068,3.060) .. (2.295,4.590);
\draw[edgeJ] (4.590,0.000) -- (4.590,2.295);
\draw[edgeJ] (4.590,2.295) -- (4.590,4.590);
\draw[edgeJ] (4.590,0.000) .. controls (4.363,1.530) and (4.363,3.060) .. (4.590,4.590);
\draw[edgeI] (1.054,0.782) -- (3.349,0.782);
\draw[edgeI] (3.349,0.782) -- (5.644,0.782);
\draw[edgeI] (1.054,0.782) .. controls (2.584,1.009) and (4.114,1.009) .. (5.644,0.782);
\draw[edgeI] (1.054,3.077) -- (3.349,3.077);
\draw[edgeI] (3.349,3.077) -- (5.644,3.077);
\draw[edgeI] (1.054,3.077) .. controls (2.584,3.304) and (4.114,3.304) .. (5.644,3.077);
\draw[edgeI] (1.054,5.372) -- (3.349,5.372);
\draw[edgeI] (3.349,5.372) -- (5.644,5.372);
\draw[edgeI] (1.054,5.372) .. controls (2.584,5.599) and (4.114,5.599) .. (5.644,5.372);
\draw[edgeJ] (1.054,0.782) -- (1.054,3.077);
\draw[edgeJ] (1.054,3.077) -- (1.054,5.372);
\draw[edgeJ] (1.054,0.782) .. controls (0.827,2.312) and (0.827,3.842) .. (1.054,5.372);
\draw[edgeJ] (3.349,0.782) -- (3.349,3.077);
\draw[edgeJ] (3.349,3.077) -- (3.349,5.372);
\draw[edgeJ] (3.349,0.782) .. controls (3.122,2.312) and (3.122,3.842) .. (3.349,5.372);
\draw[edgeJ] (5.644,0.782) -- (5.644,3.077);
\draw[edgeJ] (5.644,3.077) -- (5.644,5.372);
\draw[edgeJ] (5.644,0.782) .. controls (5.417,2.312) and (5.417,3.842) .. (5.644,5.372);
\draw[edgeI] (2.108,1.564) -- (4.403,1.564);
\draw[edgeI] (4.403,1.564) -- (6.698,1.564);
\draw[edgeI] (2.108,1.564) .. controls (3.638,1.791) and (5.168,1.791) .. (6.698,1.564);
\draw[edgeI] (2.108,3.859) -- (4.403,3.859);
\draw[edgeI] (4.403,3.859) -- (6.698,3.859);
\draw[edgeI] (2.108,3.859) .. controls (3.638,4.086) and (5.168,4.086) .. (6.698,3.859);
\draw[edgeI] (2.108,6.154) -- (4.403,6.154);
\draw[edgeI] (4.403,6.154) -- (6.698,6.154);
\draw[edgeI] (2.108,6.154) .. controls (3.638,6.381) and (5.168,6.381) .. (6.698,6.154);
\draw[edgeJ] (2.108,1.564) -- (2.108,3.859);
\draw[edgeJ] (2.108,3.859) -- (2.108,6.154);
\draw[edgeJ] (2.108,1.564) .. controls (1.881,3.094) and (1.881,4.624) .. (2.108,6.154);
\draw[edgeJ] (4.403,1.564) -- (4.403,3.859);
\draw[edgeJ] (4.403,3.859) -- (4.403,6.154);
\draw[edgeJ] (4.403,1.564) .. controls (4.176,3.094) and (4.176,4.624) .. (4.403,6.154);
\draw[edgeJ] (6.698,1.564) -- (6.698,3.859);
\draw[edgeJ] (6.698,3.859) -- (6.698,6.154);
\draw[edgeJ] (6.698,1.564) .. controls (6.471,3.094) and (6.471,4.624) .. (6.698,6.154);
\draw[edgeK] (0.000,0.000) -- (1.054,0.782);
\draw[edgeK] (1.054,0.782) -- (2.108,1.564);
\draw[edgeK] (0.000,0.000) .. controls (0.581,0.685) and (1.284,1.206) .. (2.108,1.564);
\draw[edgeK] (0.000,2.295) -- (1.054,3.077);
\draw[edgeK] (1.054,3.077) -- (2.108,3.859);
\draw[edgeK] (0.000,2.295) .. controls (0.581,2.980) and (1.284,3.501) .. (2.108,3.859);
\draw[edgeK] (0.000,4.590) -- (1.054,5.372);
\draw[edgeK] (1.054,5.372) -- (2.108,6.154);
\draw[edgeK] (0.000,4.590) .. controls (0.581,5.275) and (1.284,5.796) .. (2.108,6.154);
\draw[edgeK] (2.295,0.000) -- (3.349,0.782);
\draw[edgeK] (3.349,0.782) -- (4.403,1.564);
\draw[edgeK] (2.295,0.000) .. controls (2.876,0.685) and (3.579,1.206) .. (4.403,1.564);
\draw[edgeK] (2.295,2.295) -- (3.349,3.077);
\draw[edgeK] (3.349,3.077) -- (4.403,3.859);
\draw[edgeK] (2.295,2.295) .. controls (2.876,2.980) and (3.579,3.501) .. (4.403,3.859);
\draw[edgeK] (2.295,4.590) -- (3.349,5.372);
\draw[edgeK] (3.349,5.372) -- (4.403,6.154);
\draw[edgeK] (2.295,4.590) .. controls (2.876,5.275) and (3.579,5.796) .. (4.403,6.154);
\draw[edgeK] (4.590,0.000) -- (5.644,0.782);
\draw[edgeK] (5.644,0.782) -- (6.698,1.564);
\draw[edgeK] (4.590,0.000) .. controls (5.171,0.685) and (5.874,1.206) .. (6.698,1.564);
\draw[edgeK] (4.590,2.295) -- (5.644,3.077);
\draw[edgeK] (5.644,3.077) -- (6.698,3.859);
\draw[edgeK] (4.590,2.295) .. controls (5.171,2.980) and (5.874,3.501) .. (6.698,3.859);
\draw[edgeK] (4.590,4.590) -- (5.644,5.372);
\draw[edgeK] (5.644,5.372) -- (6.698,6.154);
\draw[edgeK] (4.590,4.590) .. controls (5.171,5.275) and (5.874,5.796) .. (6.698,6.154);
\node[vtx] at (0.000,0.000) {};
\node[vtx] at (0.000,2.295) {};
\node[vtx] at (0.000,4.590) {};
\node[vtx] at (2.295,0.000) {};
\node[vtx] at (2.295,2.295) {};
\node[vtx] at (2.295,4.590) {};
\node[vtx] at (4.590,0.000) {};
\node[vtx] at (4.590,2.295) {};
\node[vtx] at (4.590,4.590) {};
\node[vtx] at (1.054,0.782) {};
\node[vtx] at (1.054,3.077) {};
\node[vtx] at (1.054,5.372) {};
\node[vtx] at (3.349,0.782) {};
\node[vtx] at (3.349,3.077) {};
\node[vtx] at (3.349,5.372) {};
\node[vtx] at (5.644,0.782) {};
\node[vtx] at (5.644,3.077) {};
\node[vtx] at (5.644,5.372) {};
\node[vtx] at (2.108,1.564) {};
\node[vtx] at (2.108,3.859) {};
\node[vtx] at (2.108,6.154) {};
\node[vtx] at (4.403,1.564) {};
\node[vtx] at (4.403,3.859) {};
\node[vtx] at (4.403,6.154) {};
\node[vtx] at (6.698,1.564) {};
\node[vtx] at (6.698,3.859) {};
\node[vtx] at (6.698,6.154) {};
\end{tikzpicture}
\caption{The Hamming graph $K_3^{\cp,3}$.} \label{fig:3-dim-Hamming}
\end{figure}

\subsection{Even dimension} \label{subsec:even-dim}
\begin{theorem} \label{thm:Ham-even}
    If $n \ge 2 $ and $k$ is a positive even integer, we have 
    $$\mc (K_n^{\cp, k})= n^{\frac{k}{2}}.$$
\end{theorem}
\begin{proof}
    Let $G=K_n^{\cp, k}$ and $V(G)=[n]^k$. The vertices of $G$ then correspond to $k$-dimensional vectors $z=(i_1,\dots, i_k)$ with $i_j \in [n]$ for every $j \in [k]$, and the distance of two vertices in $G$ equals the number of coordinates in which they differ. Hence,  every $z \in V(G)$ has exactly $k(n-1)$ neighbors.     
    To simplify our discussion, we set $r = k/2$.
    Recall that $V_1, \dots, V_p$ are the color classes of a $\cg$-coloring of $G$ if and only if every $V_j$ induces a subgraph with 
    \begin{equation} \label{eq:min-deg-Cart-power}
        \delta(G[V_j]) \ge \frac{k(n-1)}{2}=r(n-1).
    \end{equation}
        
    We first prove that every color class $V_j$ in a $\cg$-coloring of $G$ satisfies $|V_j| \ge n^{r}$.
    Let $z_0$ be a vertex in $V_j$ and partition $V(G)$ into $L_0, \dots , L_k$ according to the distance from $z_0$. That is $L_s= \{z \in V(G) \colon \dist_G(z,z_0)=s\}$  for every $s \in [k]\cup\{0\}$. Clearly, $L_0=\{z_0\}$ and $|L_1|= k(n-1)$. We make the following general observations.
    \begin{itemize}
    \item[(a)]  If $z \in L_s$, then $N_G(z) \subseteq L_{s-1}\cup L_s \cup L_{s+1}$.
    \item[(b)] Every $z \in L_s$ has $s$ neighbors from $L_{s-1}$, for each $s \in [k]$.\\
    Indeed, $z$ differs from $z_0$ in $s$ coordinates, and a neighbor in $L_{s-1}$ can be obtained if a differing coordinate of $z$ is replaced by the corresponding coordinate of $z_0$.
    \item[(c)] Every $z \in L_s$ has $s(n-2)$ neighbors from $L_s$.\\
    Let $z=(i_1, \dots , i_k)$ and $z_0=(i_1^0, \dots , i_k^0)$. A vertex in $N_G(z) \cap L_s$ can be obtained if we choose a coordinate $i_\ell$ with $i_\ell \neq i_{\ell}^0$ and replace $i_{\ell}$ with a value $i_\ell' \in [n] \setminus \{i_\ell, i_\ell^0\}$. Since $\ell$  can be chosen in $s$ different ways such that $i_\ell \neq i_\ell^0$, and for each such $i_\ell$, there are $n-2$ possibilities to choose an appropriate $i_\ell'$, the number of neighbors of $z$ in $L_s$ is $s(n-2)$. 
    \item[(d)] If $z \in L_s \cap V_j$, then  $|N_G(z) \cap L_{s+1} \cap V_j| \ge (r-s)(n-1)$.\\
    According to~\eqref{eq:min-deg-Cart-power}, we have $|N_G(z) \cap V_j|\ge r(n-1)$. By (b) and (c), vertex $z$ has at most $s$ neighbors from $L_{s-1}\cap V_j$ and at most $s(n-2)$ neighbors from $L_s \cap V_j$. By (a), the remaining at least $ r(n-1) - s- s(n-2)= (r-s)(n-1)$ neighbors from $V_j$ all belong to $L_{s+1} \cap V_j$. It also follows that $V_j \cap L_{s+1}$ is not empty when $s+1 \leq  r$.
    \end{itemize}

    Concentrate now on the intersection of a color class $V_j$ and the sets $L_s$. We state that 
    \begin{equation} \label{eq:Ham-1}
        |V_j \cap L_s|\ge {r \choose s} (n-1)^s
    \end{equation}
    holds for every $0 \le s \le r$. The proof proceeds by induction on $s$. As $V_j \cap L_0= \{z_0\}$, the inequality \eqref{eq:Ham-1} is true for $s=0$. 

    Let $s \ge 1$ and let $m$ be the number of pairs $(z, z')$ satisfying  $z \in V_j \cap L_{s-1}$, $z' \in V_j \cap L_{s}$, and $zz' \in E(G)$. According to (d), every $z \in V_j \cap L_{s-1}$ has at least $(r-s+1)(n-1)$ neighbors in $V_j \cap L_s$. By the hypothesis, $|V_j \cap L_{s-1}|\ge {r \choose s-1} (n-1)^{s-1}$. Therefore, 
    \begin{equation} \label{eq:Ham-2}
      m \ge {r \choose s-1} (n-1)^{s-1}(r-s+1)(n-1)= (r-s+1){r \choose s-1} (n-1)^{s}.  
    \end{equation}
     Now, we consider $z'$ in these $m$ pairs. By (b), a vertex $z' \in V_j\cap  L_s$ may have at most $s$ neighbors from $V_j\cap  L_{s-1}$. This implies
     \begin{equation} \label{eq:Ham-3}
          m \leq s |V_j \cap L_s|.
     \end{equation}
      Comparing \eqref{eq:Ham-2} and \eqref{eq:Ham-3}, we may conclude that \eqref{eq:Ham-1} is true for $s$ since
      \begin{equation*} 
        |V_j \cap L_s|\ge \frac{r-s+1}{s}{r \choose s-1} (n-1)^{s}= {r \choose s} (n-1)^s.
    \end{equation*}           
      As a consequence of \eqref{eq:Ham-1}, we obtain
      \begin{equation*} 
        |V_j|\ge \sum_{s=0}^r {r \choose s} (n-1)^s =((n-1)+1)^r= n^r.
    \end{equation*}
   Since it holds for every color class in every $\cg$-coloring, 
   $$\mc(G) \leq \frac{|V(G)|}{n^r}= \frac{n^k}{n^{\frac{k}{2}}}= n^{\frac{k}{2}}.  $$
   \medskip

  To prove $\mc(G) \geq n^ {\frac{k}{2}} $, we keep the notation $r=\frac{k}{2}$. Consider the following vertex coloring $\varphi$ of $G$, where $r$-dimensional vectors represent the colors, and the color of a vertex is specified according to its first $ r$ coordinates. That is,
   $$ \varphi((i_1, \dots, i_k))= (i_1, \dots, i_{r}).
   $$
   
   According to this definition,  every vertex $z=(i_1, \dots, i_k)$ shares its color with   
   all the vertices of the form $(i_1, \dots, i_{r}, x_{r +1}, \dots , x_k)$, where  $x_{r+1}, \dots , x_k$ are arbitrary integers from $[n]$. Among them, $(k-r)(n-1)=r(n-1)$ vertices are neighbors of $z$.
   This verifies that  $\varphi$ is a $\cg$-coloring of $G$. As $\varphi$ uses $ n^ {r}$ colors, we conclude $\mc(G) \geq n^ {\frac{k}{2}}$ that establishes the theorem. 
 \end{proof}
 
 \subsection{Odd dimension} \label{subsec:odd-dim}
 
 The section starts with some simple observations concerning $\mc(K_n^{\cp, k})$ when $k$ is odd. The main result of this section shows that $\mc(K_n^{\cp, k})=\Theta( n^{\frac{k-1}{2}})$ holds for every fixed odd integer $k \ge 3$ as $n \rightarrow \infty$. However, by Theorem~\ref{thm:Ham-power-odd}, $\mc(K_n^{\cp, k})$ is strictly greater than $n^{\frac{k-1}{2}}$ whenever $k \ge 3$ is odd and $n \ge 7$.
\begin{proposition} \label{prop:Cart-power-odd}
  The following statements hold.
  \begin{itemize}
      \item[(i)] For every odd $k$, the $k$-dimensional hypercube $Q_k\cong K_2^{\cp,k}$ satisfies $$\mc(Q_k)= 2^{\frac{k-1}{2}}. $$
      \item[(ii)] For every $n \ge 3$ and odd $k$, it holds that
      $$ n^{\frac{k-1}{2}} \le \mc(K_n^{\cp,k}) \leq n^{\frac{k}{2}}.
      $$
          \end{itemize}
\end{proposition}
  \begin{proof} (i) The $k$-dimensional hypercube $Q_k$ is a $k$-regular graph and therefore, according to~\eqref{eq:k-improper}, its $\cg$-chromatic number equals its $\lfloor \frac{k}{2} \rfloor $-improper upper chromatic number. As the latter was proved to be $2^{\lfloor \frac{k}{2}\rfloor}$ in~\cite{BuSaTuPuVa-10},  the statement follows for every odd $k$.

  (ii) By definition, $K_n^{\cp,k}= K_n^{\cp,k-1} \cp K_n$. Since $k-1$ is even, Theorem~\ref{thm:Ham-even} and Proposition~\ref{prop:cart}(i) imply
  $$ \mc(K_n^{\cp,k}) \ge \mc(K_2^{\cp,k-1}) \cdot \mc(K_n) = n^{\frac{k-1}{2}} \cdot 1 =n^{\frac{k-1}{2}}.
  $$
  The upper bound is derived from Theorem~\ref{thm:Ham-even} and Proposition~\ref{prop:cart}(i) by considering the inequality
  $$ \mc(K_n^{\cp,k}) \cdot \mc(K_n^{\cp,k}) \le \mc(K_n^{\cp,2k})= n^k,
  $$
  which implies $\mc(K_n^{\cp,k}) \le n^{\frac{k}{2}}$ as stated.   
  \end{proof}
  
  \begin{corollary}
     If $n$ is a fixed positive integer and $k\rightarrow \infty$, then
      $$  \mc(K_n^{\cp,k})= \Theta(n^{\frac{k}{2}}).
      $$
  \end{corollary}
  \begin{proof}
      Proposition~\ref{prop:Cart-power-odd}(ii) and Theorem~\ref{thm:Ham-even} together imply that 
  $$ \textstyle  \frac{1}{\sqrt{n}}\, n^{\frac{k}{2}} \le \mc(K_n^{\cp,k}) \leq n^{\frac{k}{2}}  $$
  holds for every $n$ and $k$. Therefore, the statement follows.    
  \end{proof}
   Proposition~\ref{prop:Cart-power-odd}(i) together with Theorem~\ref{thm:Ham-even} demonstrates that for $n=2$ and an odd integer $k$ the equality $\mc(K_2^{\cp,k}) = \mc(K_2^{\cp,k-1}) $ is always true.
   We show that this is no longer valid when $n$ and $k$ are sufficiently large. The next result also improves the lower bound in Proposition~\ref{prop:Cart-power-odd}(ii) for every $n \ge 7$ and $k \ge 3$.
   \begin{theorem} \label{thm:Ham-power-odd}
      If $n \ge 7$ or $n=5$, and $k$ is an odd integer with $k \ge 3$, then
       $$ \mc(K_n^{\cp,k}) \ge \textstyle \frac{3}{2}\, n^{\frac{k-1}{2}}.
       $$
       \end{theorem}
   \begin{proof}
       Let $n$ and $k$ be integers that satisfy the conditions in the theorem, and let $t=\frac{k-1}{2}$. The vertices of $G=K_n^{\cp,k}$ will be represented by $k$-dimensional vectors as before. To prove the inequality, we define a $\cg$-coloring of $G$ where the colors are represented by  $(t+1)$-dimensional vectors. If $k \ge 5$, then we define the set $A$ of colors as follows.
       \begin{align*}
           A &=\{(a_1,\dots ,a_t,0) \colon a_i \in [n] \mbox{ for every } i \in [t]\}\\
           &\cup \Bigl\{(a_1, \dots , a_{t-1}, 0, a_{t+1}) \colon a_i \in [n] \mbox{ for every } i \in [t-1] \mbox{ and } \textstyle \bigl\lceil\frac{n+1}{2} \bigr\rceil+1 \leq a_{t+1} \leq n  \Bigr\}\\
           &\cup \{(a_1,\dots, a_{t-1},0,0) \colon a_i \in [n] \mbox{ for every } i \in [t-1]\}.
       \end{align*}
       If $k=3$, then $t=1$ and the above definition is not precise without further explanation. For this case, the set of colors is
       $$A'= \{(a_1,0) \colon a_1 \in [n]\} \cup
       \{(0,a_2) \colon \textstyle \bigl\lceil\frac{n+1}{2} \bigr\rceil+1 \leq a_{2} \leq n  \bigr\} \cup \{(0,0)\}.
       $$
      Hence, the colors in $A'$ for $k=3$ can be directly obtained from the general definition of $A$ by substituting $t=1$, setting $[0]= \emptyset$, and removing the first $t-1$ coordinates when referring to the vectors. Applying these conventions, we can handle the cases $k=3$ and $k \ge 5$ in a unified way in the remainder of the proof.
      
      Observe that
      $$|A|= n^t + n^{t-1} \Bigl\lfloor\frac{n-1}{2} \Bigr\rfloor +n^{t-1} \ge n^{t-1} \Bigl(n+\frac{n-2}{2}+1\Bigr)=  \frac{3}{2}\, n^{t}.$$
     A surjective mapping $\varphi \colon V(G) \rightarrow A$ is defined as
 $$\varphi((x_1,\dots, x_{k}))= 
 \begin{cases}
(x_1, \dots, x_{t-1}, x_t, 0) & \text{if } 1\le x_{t+1} \le \lceil \frac{n+1}{2}\rceil,\\[4pt]
(x_1, \dots, x_{t-1}, 0, x_{t+1}) & \text{if } 1\le x_{t} \le \lceil \frac{n+1}{2}\rceil < x_{t+1} \le n,\\[4pt]
(x_1, \dots, x_{t-1}, 0, 0) & \text{if } \lceil \frac{n+1}{2}\rceil +1 \le x_{j} \le n \text{ for every } j \in \{t,t+1\}.
\end{cases}
$$     
 By the definition of $\varphi$, every vertex of $G$ receives a unique color from $A$ and, since $n \ge 3$, every color from $A$ is used. Therefore, it suffices to show that 
 $$\deg_{G[V_i]}(z) \ge \frac{\deg_G(z)}{2}= \frac{k(n-1)}{2}  
 $$
 holds for every $z \in V(G)$ and for the color class $V_i$ that contains $z$. Let $z=(x_1, \dots, x_k)$ and $z \in V_i$. By the definition of $\varphi$, vertices  $z$ and $z'$ always receive the same color if they share the first $t+1$ coordinates $x_1, \dots, x_{t+1}$. This way, by changing one of the coordinates $x_{t+2}, \dots, x_k$, we can identify $(k-t-1)(n-1)=t(n-1)$ neighbors of $z$ in the same color class $V_i$. 
 To show additional vertices in $N_G(z) \cap V_i$, we consider three cases based on the three lines in the definition of $\varphi$.
 \begin{description}
     \item[Case 1. $1\le x_{t+1} \le \lceil \frac{n+1}{2}\rceil$] 
     In this case, every vertex $w=(x_1, \dots , x_t, y, x_{t+2}, \dots, x_k)$ with $y\neq x_{t+1}$ and $1\le y \le \lceil \frac{n+1}{2}\rceil$ receives the same color $(x_1, \dots , x_t,0)$ as $z$. We may therefore conclude
     \begin{equation} \label{eq:thm6}
         \deg_{G[V_i]}(z) \ge t(n-1) + \Bigl\lceil \frac{n+1}{2} \Bigr\rceil -1 \ge \Bigl(t+ \frac{1}{2} \Bigr) (n-1) =\frac{k(n-1)}{2}.
     \end{equation}
     \item[Case 2. $1\le x_{t} \le \lceil \frac{n+1}{2}\rceil < x_{t+1} \leq n$]    
     By the definition of $\varphi$, every vertex $(x_1, \dots , x_{t-1}, y,  x_{t+1}, \dots, x_k)$ with $y\neq x_{t}$ and $1\le y \le \lceil \frac{n+1}{2}\rceil$  receives the same color $(x_1, \dots , 0,x_{t+1})$ as $z$. We obtain the same inequality~\eqref{eq:thm6} here as for Case 1.
      \item[Case 3. $\lceil \frac{n+1}{2}\rceil +1 \le x_{t} \leq n$ \rm{and} $\lceil \frac{n+1}{2}\rceil +1 \le x_{t+1} \leq n$]    
     In this case, $\varphi(z)=(x_1, \dots , x_{t-1}, 0,0)$ and the same color is assigned to the vertices obtained by changing the $t$th  or the $(t+1)$st coordinate of $z$ to a different value from $\{\lceil \frac{n+1}{2}\rceil +1, \dots, n\}$.  This way, we can identify $2(n-\lceil \frac{n+1}{2}\rceil-1)$ new neighbors having the same color as $z$. The conclusion is
     \begin{align*} 
         \deg_{G[V_i]}(z) &\ge t(n-1) + 2\Bigl(n-\Bigl\lceil \frac{n+1}{2}\Bigr\rceil-1\Bigr)\\
         &= t(n-1) + 2\Bigl\lfloor \frac{n-1}{2}\Bigr\rfloor-2\\
          &\ge \frac{k(n-1)}{2},
     \end{align*}
     where the last inequality can be verified separately for the odd values of $n$ with $n \ge 5$ and for the even values with $n \ge 8$.  
 \end{description}
 We have proved that $\varphi$ is a $\cg$-coloring of $G= K_n^{\cp,k}$ that uses $|A| \ge \frac{3}{2} n^t$ colors and hence, $\mc(K_n^{\cp,k}) \ge \frac{3}{2} n^t$ holds under the conditions given in the theorem. 
   \end{proof}

 \begin{remark} \label{rmk:thm7}
     In the above theorem, our goal was to show that the lower bound $n^{\lfloor \frac{k}{2}\rfloor} \le \mc(K_n^{\cp, k})$ can be exceeded if $k$ is odd and $n$ is large enough. However, we do not think that the bound in Theorem~\ref{thm:Ham-power-odd} is always tight. In fact, if $n$ is large enough, the color classes represented by $(x_1, \dots, x_{t-1},0,0)$ can be split into smaller color classes.
 \end{remark} 
While Theorem~\ref{thm:Ham-power-odd} demonstrates that the equality $\mc(K_n^{\cp, k})= n^{\lfloor\frac{k}{2}\rfloor}$ is not true if $k \ge 3$ is odd and $n \ge 7$, the main theorem in this subsection shows that the asymptotic result $\mc(K_n^{\cp, k})= \Theta(n^{\lfloor\frac{k}{2}\rfloor})$ holds for every fixed $k$ as $n\rightarrow \infty$. In view of Theorem~\ref{thm:Ham-even} and since $ \mc(K_n^{\cp, 1})=1$ is known, it is enough to concentrate on the odd values of $k$ with $k \ge 3$.
\begin{theorem} \label{thm:asymptotic}
    If $k\ge 3$ is a fixed odd integer and $n \rightarrow \infty$, then
    $$ \mc(K_n^{\cp, k})=\Theta( n^{\frac{k-1}{2}}) \,.
    $$
\end{theorem}
  \begin{proof}
     In view of Theorem~\ref{thm:Ham-power-odd}, it suffices to prove that a constant $C^*(k)$ depending only on $k$ exists for every odd $k$ such that 
     \begin{equation} \label{eq:asymptotic-1}
         \mc(K_n^{\cp, k}) \leq C^*(k)\, n^{\frac{k-1}{2}} 
    \end{equation} 
     holds for every $n$. To verify~\eqref{eq:asymptotic-1}, we establish a lower bound on the sizes of the color classes in a $\cg$-coloring of $K_n^{\cp, k}$.
     We have no intention to optimize $C^*(k)$, for instance at one step we shall apply the rather rough estimate $n-1\geq n/2$, and even that will do the job as regards the claimed growth order.

     Let $k \ge 3$ be an odd integer and let $V_s$ be a color class in a $\cg$-coloring of $G=K_n^{\cp, k}$. Let us set $t=\frac{k-1}{2}$ and $q=n-1$. The vertices of $G$ are represented by $k$-dimensional vectors; that is, $V(G)=[n]^k$. A neighbor of a vertex $v \in V(G)$ can be obtained by changing exactly one coordinate in $v$. Let $X_i(v)$ denote the set of those neighbors of $v$ that differ only in the $i$th coordinate from $v$. Then $N_G(v)=\bigcup_{i=1}^k X_i(v)$ and $G[X_i(v)] \cong K_{q}$ for each $i \in [k]$. Moreover, there are no edges between $X_i(v)$ and $X_{i'}(v)$ if $i \neq i'$.

      We say that a vertex $z$ from the color class $V_s$ is \emph{dense in coordinate $i$} if $X_i(z)$ contains at least $\frac{q}{2t+2}$ vertices from $V_s$. Let $D(z)$ be the set of dense coordinates of $z$. Formally,  
      $$D(z)= \left\{ i \in [k] \colon |X_i(z) \cap V_s| \ge \frac{q}{2t+2}\right\}.
      $$

      We first show that every $z \in V_s$ has at least $t+1$ dense coordinates. Suppose, to the contrary, that $|D(z)|\leq t$. Then, by counting the neighbors of $z$ from the same color class $V_s$, we get 
      $$ |N_G(z) \cap V_s| < t q + (t+1) \frac{q}{2t+2} = \frac{(2t+1)q}{2} = \frac{\deg_G(z)}{2}
      $$
       that contradicts the majority condition~\eqref{eq:majority-1} for $z$. Therefore, $|D(z)|\geq t+1$.
        \medskip
        
       In the next part of the proof, we consider an arbitrary vertex $z_0$ from the color class $V_s$ and, for every $j \leq t+1$, identify a set $Z_j$ of vertices at a distance $j$ from $z_0$ such that all of them belong to the same color class $V_s$. Then, we use the last set $Z_{t+1}$ to establish a lower bound on the size of $V_s$.
       \medskip

       Recall the following fact. If we have any $p$ not necessarily distinct subsets of a $(2t+1-j)$-element set $H$ such that each subset contains at least $t+1-j$ elements, then there is an element of $H$ that occurs in at least $\frac{t+1-j}{2t+1-j}\cdot p$ subsets. To simplify the notation, we introduce the constant 
       $$c_j=\frac{t+1-j}{2t+1-j},$$ for $j \in [t]$, to denote this ratio.
       \begin{itemize}
           \item[(1)] By the majority condition~\eqref{eq:majority-1}, $z_0$ has at least $\frac{(2t+1)q}{2}$ neighbors in $V_s$. These neighbors are distributed among the sets $X_1(z_0), \dots, X_{2t+1}(z_0)$. Consequently, there is an index $k_1$ such that $|X_{k_1} \cap V_s| \ge \frac{q}{2}$. We define $Z_1= X_{k_1} \cap V_s $ that ensures $|Z_{1}| \ge \frac{q}{2} $ and $Z_1 \subseteq V_s$.
           \item[(2)] For each vertex $z \in Z_{1}$, we consider the set $D(z) \setminus \{k_1\}$. Each such set contains at least $t$ elements from the $2t$-element index set $\{1, \dots, 2t+1\}\setminus \{k_1\} $. Thus, there is an index $k_2 \in \{1, \dots, 2t+1\}\setminus \{k_1\} $ that occurs in $D(z) \setminus \{k_1\}$ for at least $c_1 |Z_{1}|$ different vertices $z \in Z_{1}$. Let
           $$Z_{1}'= \{z \in Z_{1} \colon k_2 \in D(z) \setminus \{k_1\}\}.
           $$
           Then $|Z_{1}'| \ge c_1 |Z_{1}|$.

           Now, we define 
           $$ Z_2= \bigcup_{z \in Z_1'} (X_{k_2}(z) \cap V_s). 
           $$
           By definition, $Z_2 \subseteq V_s$. Further, each vertex in $Z_2$ is obtained from a vertex in $Z_1'$ by changing the $k_2$th coordinate. Since $k_2 \neq k_1$, all vertices in $Z_2$ are at distance two from $z_0$. Since any two different vertices $z'$ and $z''$ from $Z_1'$ differ in their $k_1$th coordinate, the sets $X_{k_2}(z')$ and $X_{k_2}(z'')$ are disjoint. By the determination of $Z_1'$, every vertex $z \in Z_1'$ is dense in coordinate $k_2$, and hence $|X_{k_2}(z) \cap V_s|\ge \frac{q}{2t+2}$. Therefore, we conclude
           \begin{equation} \label{eq:asympotic-2}
               |Z_2| \ge \frac{q}{2t+2} |Z_1'| \ge c_1\, \frac{q}{2t+2} |Z_1| \ge \frac{c_1}{2(2t+2)}\, q^2.
           \end{equation}           
           \item[(j)] If $t \ge 3$, then in step $j$, for $3 \le j \le t+1$, we start with a set $Z_{j-1}\subseteq V_s$ so that each vertex in $Z_{j-1}$ was obtained from $z_0$ by changing the coordinates $k_1, \dots, k_{j-1}$. For each $z \in Z_{j-1}$, we consider the set $D(z) \setminus \{k_1, \dots, k_{j-1}\}$. Since these are all subsets of $[k]\setminus \{k_1, \dots, k_{j-1}\}$, there is an index $k_j$ that is repeated in at least $c_{j-1} |Z_{j-1}|$ such subsets. We now define 
           $$Z_{j-1}'= \left\{z \in Z_{j-1} \colon k_j \in D(z) \setminus \{k_1, \dots, k_{j-1}\}\right\}
           $$
           and also
           $$ Z_j= \bigcup_{z \in Z_{j-1}'} (X_{k_j}(z) \cap V_s). 
           $$
           By construction, any two elements of $Z_j$ differ in coordinate $k_j$, or in coordinate $k_{j-1}$, or in both coordinates. As $|Z_{j-1}'| \ge c_{j-1} |Z_j|$ and $k_j$ is a dense coordinate for every $z \in Z_{j-1}'$, we infer
           \begin{equation} \label{eq:asymptotic-3}
              |Z_j| \ge \frac{q}{2t+2} |Z_{j-1}'| \ge  c_{j-1}\, \frac{q}{2t+2} |Z_{j-1}|. 
           \end{equation}                               
       \end{itemize}
     By performing steps ($1$), ($2$), and ($j$) for  $j=3,\dots, t+1 $, we obtain a set $Z_{t+1} \subseteq V_s$. Inequalities~\eqref{eq:asympotic-2} and \eqref{eq:asymptotic-3} imply that
     \begin{align*}
         |Z_{t+1}|  &\ge \frac{q}{2t+2} \,|Z_{t}'| \ge  \frac{c_{t}\, q}{2t+2} \, |Z_{t}|
          \ge  \frac{q}{2}\, \left(\frac{q}{2t+2}\right)^t \, \prod_{j=1}^t c_j =
          \frac{\prod_{j=1}^t c_j}{2^{t+1} (t+1)^t}\, q^{t+1}.
     \end{align*}
  As $Z_{t+1} \subseteq V_s$ and $q=n-1 \ge \frac{1}{2}n$, it implies
  \begin{equation}  \label{eq:asymptotic-4}
      |V_s| \ge  |Z_{t+1}| \ge \frac{\prod_{j=1}^t c_j}{2^{2t+2} (t+1)^t}\, n^{t+1}.
  \end{equation} 
 Since the lower bound in~\eqref{eq:asymptotic-4} holds for each color class in a $\mc$-coloring of $G$, we may conclude that
 $$\mc(G) \leq C^*(k) \frac{n^k}{n^{t+1}}= C^*(k) \, n^{\frac{k-1}{2}},
 $$
where the constant may be chosen as 
$$C^*(k)= \frac{2^{k+1} (\frac{k+1}{2})^{\frac{k-1}{2}}}{\prod_{j=1}^{\frac{k-1}{2}} c_j}= 
\frac{2^{\frac{k+3}{2}} (k+1)^{\frac{k-1}{2}}}{\prod_{j=1}^{\frac{k-1}{2}} c_j}.
$$

 The final conclusion is that $\mc(K_n^{\cp, k})=\Theta( n^{\frac{k-1}{2}})$ holds for every odd integer $k \ge 3$ as $n \rightarrow\infty.$
  \end{proof}

  The two results above - Theorem~\ref{thm:Ham-even} and Theorem~\ref{thm:asymptotic} - reveal a notable difference between even and odd dimensions. For even $k$, the majority C-chromatic number is exactly $n^{k/2}$, whereas for odd $k\ge3$ its order of growth is $n^{(k-1)/2}$.
 \medskip

  At the end of this section, we put some further remarks on the sharpness of the general lower bounds given in Proposition~\ref{prop:cart}.
\begin{remark} \label{rmk:cart-tightness}
 By Theorem~\ref{thm:Ham-even}, the equality $\mc(K_n^{\cp, k_1 +k_2})= \mc(K_n^{\cp, k_1})\, \mc(K_n^{\cp, k_2})$ holds if $k_1$ and $k_2$ are even. By Proposition~\ref{prop:Cart-power-odd}(i), the same is true if $n=2$, no matter the parities of the exponents.
 These provide infinitely many examples for the tightness of the lower bound in Proposition~\ref{prop:cart}(i). On the other hand, Theorem~\ref{thm:2-dim-Ham} shows that the difference $\mc(G \cp H)- \mc(G)\, \mc(H)$ can be arbitrarily large. 

    Concerning Proposition~\ref{prop:cart}(ii), the tightness is demonstrated directly by Theorem~\ref{thm:2-dim-Ham}. Indeed, if $m\ge n$, then $\delta(K_m) \ge \Delta(K_n)$ and $\mc(K_m \cp K_n)= n$. In contrast, 
    we have infinitely many examples of the form $G=K_2^{\cp, k}$ and $H=K_2^{\cp, 2}$ such that, by Theorem~\ref{thm:Ham-even} and Proposition~\ref{prop:Cart-power-odd}(i), the difference
    $\mc(G \cp H)- m_H \ge 2^{\frac{k+1}{2}}-4$ can be arbitrarily large and the condition in Proposition~\ref{prop:cart}(ii) is satisfied. 
    \end{remark}

\section{Cartesian grids} \label{sec:grids}
Our goal in this section is to determine the $\cg$-chromatic number of the grid $P_m \cp P_n$ for every $m$ and $n$. 
We organize our results according to the parity of $m$ and $n$.

\subsection{Both $m$ and $n$ are even} 
\label{subsec:grids-even-even}
The following proposition settles the problem for grids when both parameters are even. With this result, we obtain further tight examples for the lower bound in Proposition~\ref{prop:cart}(i).
\begin{proposition} \label{prop:grid-1}
    If $m$ and $n$ are even integers with $m \ge 4$ and $n \ge 4$, then
    $$ \mc(P_m \cp P_n)= \frac{mn}{4}.
    $$     
\end{proposition}
  \begin{proof}
      As it was shown in~\cite{majority-1}, $\mc(P_k)=  k/2$ holds when $k$ is even.  Let $m$ and $n$ be even numbers, and let $G$ be the grid $P_m \cp P_n$. By Proposition~\ref{prop:cart}(i), $\mc(G) \ge \frac{1}{4}mn$.

      Consider a $\mc$-coloring $\varphi$ of $G$ and the color classes $V_1, \dots , V_\ell$ associated with $\varphi$. We claim that $|V_i| \ge 4$ holds for every $i \in [\ell]$. 
      
      Consider an arbitrary vertex $v$ that belongs to a color class $V_i$. By the majority condition~\eqref{eq:majority-1}, $\deg_{G[V_i]}(v) \ge 2$ holds if $v$ is an inner or a border vertex of the grid; and $\deg_{G[V_i]}(v) \ge 1$ if $v$ is a corner vertex. If $V_i$ contains at most one corner vertex of $G$, then there is at most one leaf in $G[V_i]$. Consequently, the subgraph $G[V_i]$ contains a cycle that ensures $|V_i| \ge 4$. If $V_i$ contains at least two corner vertices, $G[V_i]$ may contain a cycle and $|V_i| \ge 4$ follows as before, or it may be a forest. For this latter case, since  $m,n \ge 4$ and  $V_i$ induces a connected graph by Observation~\ref{obs:basic}(i), the inequality $|V_i| \ge 4$ remains true. It follows that each color class contains at least four vertices and $\mc(G) \leq \frac{1}{4}mn.$ Consequently, the equality holds. 
      
  \end{proof}

\begin{remark}
    It is straightforward to check that $\mc(P_2 \cp P_2)=2$. Now, let $m\ge 4$  be even and consider the grid $P_m\cp P_2$. 
   Observe first that $\mc(P_m \cp P_2) \ge \frac{1}{2}m+1$. Indeed, a $\cg$-coloring can be defined where the first two color classes correspond to the layers $^1P_2$ and $^mP_2$,  and the remaining color classes are each of size four. As for the reverse inequality, we can prove it along the lines of the proof of Proposition~\ref{prop:grid-1} and show that any color class different from $^1P_2$ and $^mP_2$ contains at least four vertices.    
    Thus the formula for an even $m$ is $\mc(P_m \cp P_2)= \frac{1}{2}m+1$.
\end{remark}

    
\subsection{At least one of $m$ and $n$ is odd} \label{subsec:grids-odd}

\begin{theorem} \label{thm:grid-odd}
 If $m \ge 3$,  $n\ge 3$, and at least one of $m$ and $n$ is odd, then
   \begin{equation} \label{eq:grid-thm-odd-odd}
       \mc(P_m \cp P_n)=   1 + \Bigl\lfloor\frac{m}{2}\Bigr\rfloor\Bigl\lfloor\frac{n}{2}\Bigr\rfloor.
   \end{equation} 
\end{theorem} 
\begin{proof}
    Let $G=P_m \cp P_n$ where at least one of the integers $m\ge 3$ and $n\ge 3$ is odd. We set $V(P_m)=[m]$ and $V(P_n)=[n]$. A $\cg$-coloring of $G$ with $1+\lfloor\frac{m}{2}\rfloor \lfloor\frac{n}{2}\rfloor$ colors can be defined as follows.
    If both $m$ and $n$ are odd, then let the color classes be
    \begin{align} \label{eq:grid-col-classes-2-odd-odd}
      \nonumber  V_0 &= \{(i, n) \colon i \in [m] \} \cup \{(m,j) \colon j \in [n-1]\}
      \mbox{ and}\\
        V(p,q) &= \{2p-1,2p\} \times \{2q-1,2q\},
        \end{align}
    for $p \in [\frac{m-1}{2}]$ and $q \in [\frac{n-1}{2}]$. 
    If $m$ is odd and $n$ is even, we define the color classes as
     \begin{align} \label{eq:grid-col-classes-2-odd-even}
      \nonumber  V_0' &= \{(m,j) \colon j \in [n]\}
      \mbox{ and}\\
        V'(p,q) &= \{2p-1,2p\} \times \{2q-1,2q\},
        \end{align}
     for $p \in [\frac{m-1}{2}]$ and $q \in [\frac{n}{2}]$. 
    As the Cartesian product is commutative, these two types of colorings establish the lower bound $1+\lfloor\frac{m}{2} \rfloor \lfloor\frac{n}{2} \rfloor \leq \mc(G)$ for all cases when at least one of $m$ and $n$ is odd.
  \medskip
  
    For the reverse inequality, we first settle the case when a color class contains opposite corner vertices.
    \begin{claim} \label{claim:odd-odd-1}
        If there exists a $\mc$-coloring of $G$ so that a color class $V_i$ contains two opposite corner vertices, then $\mc(G) \leq 1+\lfloor\frac{m}{2} \rfloor \lfloor\frac{n}{2} \rfloor $.
    \end{claim}
    \textit{Proof.} Let $\mc(G)=t$. If $V_i$ contains two opposite corner vertices, the connectivity of $G[V_i]$ implies $|V_i| \ge m+n-1$. Each remaining color class $V_j$ includes at most one corner vertex. Therefore, $V_j$ contains a cycle and $|V_j|\ge 4$, for each $j \in [t]\setminus \{i\}$. It follows that
    $ 4(t-1) +(m+n-1) \le mn
    $, and we may conclude that     
    $$ \mc(G)=t \le 1+ \frac{mn-m-n+1}{4}= 1+ \frac{(m-1)(n-1)}{4}\le 
    1 + \Bigl\lfloor\frac{m}{2}\Bigr\rfloor\Bigl\lfloor\frac{n}{2}\Bigr\rfloor 
    $$
    holds as stated. \smallqed
    \medskip
    
     By Claim~\ref{claim:odd-odd-1}, it remains to consider the case when every color class contains either at most one corner vertex, or exactly two corner vertices such that these are non-opposite corner vertices. Since the Cartesian product is commutative, and we did not put conditions on the individual parity of $m$ or $n$, the parameters $m$ and $n$ are interchangeable in the theorem. Furthermore, the vertices of $P_m$ and $P_n$ may be renamed in reverse order. Therefore, whenever we consider two (non-opposite) corner vertices from a color class, we may suppose (for the local properties) that these two corner vertices are $(m,1)$ and $(m,n)$. If only one color class is from the set $\{^1P_n, ^mP_n, P_m^1, P_m^n\}$, we will assume that it is $^mP_n$. If we have two color classes from this set, we may assume that these are $^1P_n$ and $^mP_n$.
     
     Before proceeding to further claims, we introduce the terminology used in the proof. We say that the vertex set of the layer $^j P_n$ is an \emph{odd} or an \emph{even column} if $j$ is odd or even, respectively. Thus, no matter whether $m$ is odd or even, we have $\lceil\frac{m}{2}\rceil$ odd and $\lfloor\frac{m}{2}\rfloor$ even columns.  A color class $V_i$ is of \emph{type-1} if $V_i \in \{^1P_n, ^mP_n\}$ and it is of \emph{type-2} otherwise. 

     If a type-2 color class includes at most one corner vertex from the grid, then $G[V_i]$ has at most one leaf and consequently, it contains a cycle. If a type-2 color class includes two corner vertices, then it either contains a cycle or it is a path between two (non-opposite) corner vertices.

    Let $\varphi$ be a $\mc$-coloring of $G$ with color classes $V_1, \dots , V_t$. Hence $t=\mc(G)$. By Observation~\ref{obs:basic}(i), every $V_i$, for $i \in [t]$, induces a connected subgraph in $G$. Let $T_1$ and $T_2$ denote the family of type-1 and type-2 color classes among $V_1, \dots, V_t$, respectively.
    \begin{claim} \label{claim:odd-odd-2}
       Let $V_i \in T_2$. Then there exists an even column $^j\!P_n$ such that $|^j\!P_n \cap V_i|\ge 2$. Moreover, there is an odd column $^{j'}\!P_n$ such that $|^{j'}\!P_n \cap V_i|\ge 2$.
    \end{claim}
    \textit{Proof.} 
    Assume first that $G[V_i]$ contains a cycle $C$.  Recall that it can be assumed, without loss of generality, that $V_i \cap \{(1,1),(1,n)\}=\emptyset$. Let $x_1$ be the minimum integer so that $C$ contains a vertex from $^{x_1}\!P_n$, and let $x_2$ be the maximum integer so that $(x_1, x_2) \in V(C)$. Then $v=(x_1, x_2)$ is the upper leftmost vertex of $C$. By our assumptions, $v$ is not a corner vertex in the grid. By the definition of the upper leftmost vertex, its neighbors $(x_1-1, x_2)$ and $(x_1, x_2+1)$ (if they exist), do not belong to $C$. Therefore, the remaining two neighbors of $v$, namely $u=(x_1, x_2-1)$ and $w=(x_1+1,x_2)$, do exist and belong to $C$. 

    Now, consider the minimum integer $x_2'$ so that $(x_1, x_2') \in V(C)$. Then $v'=(x_1, x_2')$ is the bottom leftmost vertex of the cycle. Note that $v'$ is not a corner vertex in the grid and $v'\neq v$, but $v'=u$ is possible. Since $v'$ is the bottom leftmost vertex, $(x_1-1, x_2')$ and $(x_1, x_2'-1)$ do not belong to $V(C)$. Hence the cycle must contain the remaining two neighbors of $v'$, namely $u'=(x_1, x_2'+1)$ and $w'=(x_1+1, x_2')$. Since $x_2 \neq x_2'$, the four vertices $v, v', w,w'$ are pairwise different. Further, $v, v' \in\, ^{x_1}\!P_n$ and $w,w' \in\, ^{(x_1+1)}\!P_n$. We conclude that $V_i$ contains at least two vertices from each of two consecutive coloumns. As one of them is an odd coloumn and the other one is even, the statement follows.
    \medskip
    
    In the other case, $V_i$ is a type-2 color class without a cycle. Thus, $G[V_i]$ is a path $P$ between two corner vertices. We may assume that the end vertices of $P$ are $(m,1)$ and $(m,n)$. Let $y$ be the maximum integer with  $(m,y) \notin V_i$.  As $V_i \neq\, ^m\!P_n$, such an integer exists and $v=(m,y+1) \in V_i$. No matter whether $y+1=n$ or $y+1 <n$, the majority condition~\eqref{eq:majority-1} implies that $w=(m-1, y+1) \in V_i$. Now, consider the minimum integer $y'$ such that $(m,y')\notin V_i$. By definition, $v'=(m,y'-1) \in V_i$. Then, applying the majority condition~\eqref{eq:majority-1} again, we conclude that $w'=(m-1, y'-1) \in V_i$. Since $v,v',w,w'$ are four different vertices, we infer that $V_i$ contains at least two vertices from an odd column and further two vertices from an even column. This finishes the proof of the claim.    
    \smallqed
    \medskip
    
     \begin{claim} \label{claim:odd-odd-3}
      $|T_2| \leq \lfloor\frac{m}{2} \rfloor \lfloor\frac{n}{2} \rfloor$.
    \end{claim}
    \textit{Proof.} 
    We define the following family of index pairs:
    $$R=\{(i,j) \colon j \mbox{ is even and } |^j\!P_n \cap V_i| \ge 2\}.
    $$    
    By Claim~\ref{claim:odd-odd-2}, every color class $V_i \in T_2$ can be associated with an even coloumn $^j\!P_n$ such that $|^j\!P_n \cap V_i| \ge 2$. This implies $|T_2|\le |R|$. Further,  the grid contains $\lfloor\frac{m}{2}\rfloor$ even columns and therefore, the second entry in a pair $(x,y) \in R$ can be chosen from $\lfloor\frac{m}{2}\rfloor$ different integers. As $|^j\!P_n|=n$, each such integer may occur in at most $\lfloor\frac{n}{2}\rfloor$ pairs. Consequently, $|R| \leq \lfloor\frac{m}{2} \rfloor \lfloor\frac{n}{2} \rfloor $. The two estimations on the size of $R$ directly imply the statement.
      \smallqed
    \medskip

     \begin{claim} \label{claim:odd-odd-4}
      If $|T_1| \leq 1$, then $t \leq 1+ \lfloor\frac{m}{2} \rfloor \lfloor\frac{n}{2} \rfloor$.
    \end{claim}
    \textit{Proof.} As $t=|T_1|+|T_2|$, Claim~\ref{claim:odd-odd-3} directly implies the inequality.
    \smallqed
    \medskip
       
    Now, in view of Claim~\ref{claim:odd-odd-4}, it is enough to consider the $\mc$-colorings of $P_m \cp P_n$ with $|T_1|=2$. Recall that we may assume, without loss of generality, that  $T_1= \{^1\!P_n, ^m\!P_n\}.$    
    Deleting the two layers $^1\!P_n$ and $^m\!P_n$ from $G=P_m \cp P_n$, the color classes in $T_2$ partition   the grid $P_{m-2} \cp P_{n}$ into $t-2$ classes.  This grid contains the columns $^2\!P_n, \dots, ^{(m-1)}\!P_n$. Among them, there are $\lfloor \frac{m-2}{2}\rfloor$ odd columns.\footnote{Remark that the case of $m=3$ is not possible here, as then
    $T_1=\{^1\!P_n, ^3\!P_n\}$ would violate the majority condition~\eqref{eq:majority-1} for the border vertex $(2,1)$.} As every color class of type-2 contains at least two vertices from the same odd column, we have $|T_2| \leq \lfloor \frac{m-2}{2}\rfloor \lfloor \frac{n}{2}\rfloor$. Now, we can conclude the proof by the following calculation:   
    \begin{align*}
        \mc(P_m \cp P_{n})-2 =|T_2| 
         &\leq  \Bigl\lfloor\frac{m-2}{2}\Bigr\rfloor \Bigl\lfloor\frac{n}{2} \Bigr\rfloor\\
        &= \Bigl\lfloor\frac{m}{2}\Bigr\rfloor \Bigl\lfloor\frac{n}{2} \Bigr\rfloor -  \Bigl\lfloor\frac{n}{2}\Bigr\rfloor\\
         &\leq \Bigl\lfloor\frac{m}{2}\Bigr\rfloor \Bigl\lfloor\frac{n}{2} \Bigr\rfloor -1.
    \end{align*}
     This implies the required result for all pairs of  integers $m\ge 3$ and $n \ge 3$ when at least one of them is odd, and completes the proof of the theorem. 
   \end{proof}



\section{Concluding remarks} \label{sec:concluding}
This paper is the first step in the study of the majority C-chromatic number of Cartesian product graphs. Therefore, many questions remain open, for example, investigating the behavior of $\mc(G \cp H)$ over specific graph classes. Here, we discuss open problems closely related to the topics of this paper.
\medskip

Proposition~\ref{prop:cart}(i) establishes a basic lower bound on $\mc(G \cp H)$. Its sharpness is pointed out in Remark~\ref{rmk:cart-tightness}. However, a complete characterization of such graph pairs attaining equality in Proposition~\ref{prop:cart}(i) is not provided in this work, so we propose it as an open problem.  

\begin{open}
    Characterize the graph pairs $(G,H)$ that satisfy the equality $\mc(G\cp H) = \mc(G)\, \mc(H)$.
\end{open}

Theorem~\ref{thm:2-dim-Ham} determines the $\cg$-chromatic number of two-dimensional Hamming graphs as $\mc(K_m \cp K_n) = \min\{m,n\}$. A natural question is to ask for formulas for Hamming graphs of higher dimension. As shown by Theorem~\ref{thm:Ham-power-odd}, the following lower bounds do not always hold with equality. However, the bound in (ii) gives the exact value of $\mc(G)$ if $G= K_n^{\cp, k}$ and $k$ is even, or if $n=2$. For other cases, the result
may provide a stepping stone to discover the exact values.   
\begin{proposition} \label{prop:imbalanced}
    Let $G= K_{n_1} \cp \cdots \cp K_{n_d}$ be a $d$-dimensional Hamming graph such that $1 \leq n_1 \leq \dots \leq n_d$ holds for the orders of the factors.
    \begin{itemize}
        \item[(i)] If $r \leq d-1$ is an integer such that $\sum_{i=1}^{r}(n_i-1) \leq \sum_{i=r+1}^{d}(n_i-1)$, then
        $$ \mc(G) \ge \prod_{i=1}^{r}n_i.
        $$
         \item[(ii)] It holds that
         $$ \mc(G) \ge \prod_{i=1}^{\lfloor d/2\rfloor}n_i.
        $$         
    \end{itemize}
\end{proposition}
\begin{proof}
   (i) Let $G$ be the $d$-dimensional Hamming graph $K_{n_1} \cp \dots \cp K_{n_d}$ with $V(G)=[n_1] \times \dots \times [n_d]$. 
   Assume that $r$ satisfies the condition in the statement. We define the following vertex coloring of $G$, where the colors are represented by $r$-dimensional vectors from the set $A=[n_1]\times \dots \times [n_r]$. For every vertex $x=(x_1, \dots , x_d)$ from $G$, let
   $$ \varphi(x)= (x_1, \dots, x_{r}).
   $$  
   We first prove that $\phi$ is a $\cg$-coloring of $G$. Let $x=(x_1, \dots , x_d)$ be an arbitrary vertex from $V(G)$. It shares its color $(x_1, \dots , x_r)$ with every neighbor $z=(z_1, \dots , z_d)$ if their coordinates differ in only one place such that $x_\ell \neq z_\ell$ and $r+1 \leq \ell \leq d$. Therefore, the color class of $x$ contains exactly $\sum_{i=r+1}^{d}(n_i-1)$ vertices from $N_G(x)$. Since $\deg_G(x)= \sum_{i=1}^{d}(n_i-1)$, and the condition ensures $\sum_{i=1}^{r}(n_i-1) \leq \sum_{i=r+1}^{d}(n_i-1)$, the coloring $\varphi$ satisfies the majority condition~\eqref{eq:majority-1} for $x$. As it is true for every vertex of $G$, $\varphi$ is a $\cg$-coloring of $G$. 
   
   It is clear that $\varphi$ uses $|A|=\prod_{i=1}^{r}n_i$ colors. Therefore, we may conclude that
   $$ \mc(G) \ge \prod_{i=1}^{r}n_i.
        $$
   (ii) The statement in (ii) is a consequence of (i). Indeed, by the nondecreasing order $1\le n_1 \leq \dots \leq n_k$,
   \begin{align*}
       \sum_{i=1}^{\lfloor d/2 \rfloor}(n_i-1) &= 
       \Bigl(\sum_{i=1}^{\lfloor d/2 \rfloor}n_i\Bigr) - \Bigl\lfloor \frac{d}{2} \Bigr\rfloor
       \leq \Bigl(\sum_{i=\lfloor d/2 \rfloor +1}^{d}n_i\Bigr) - \Bigl\lceil \frac{d}{2} \Bigr\rceil\\
       &= \sum_{i=\lfloor d/2 \rfloor +1}^{d}(n_i-1).
   \end{align*} 
 The above inequality clearly holds if $d$ is even, and $1 \le n_d$ verifies it for the case when $d$ is odd. Hence the condition in (i) is satisfied if $r=\lfloor d/2 \rfloor$, and (i) directly implies the lower bound in (ii).  
\end{proof}
As the lower bounds in Proposition~\ref{prop:imbalanced} do not always give the exact values, and since the lower bound in Theorem~\ref{thm:Ham-power-odd} also allows improvement (see Remark~\ref{rmk:thm7}), we pose the following open problems.

\begin{open}
    Determine the $\cg$-chromatic number of three- and four-dimensional (imbalanced) Hamming graphs.
\end{open}
\begin{open}
    Determine $\mc(K_n^{\cp, k})$ if $k\ge 3$ is odd and $n \ge 3$.
\end{open}
\medskip

\section*{Acknowledgements}
Cs.\ Bujt\'as was supported by the Slovenian Research and Innovation Agency (ARIS) under the grants P1-0297, N1-0355, and J1-70045. 
M.\ Dettlaff and H.~Furma\'nczyk were supported by the European Commission's Horizon Europe Research and Innovation programme through the Marie Skłodowska-Curie Actions Staff Exchanges (MSCA-SE) under Grant Agreement no.101182819 (COVER: (C)ombinatorial (O)ptimization for (V)ersatile Applications to (E)mer\-ging u(R)ban Problems).
Research of Zs.\ Tuza was supported in part by the ERC Advanced Grant ``ERMiD'' no.101054936.


\end{document}